\documentclass[a4paper,11pt,twoside]{article}
\usepackage[left=3cm, right=3cm, top=3.5cm, bottom=3.5cm]{geometry}
\usepackage{type1cm}
\usepackage[utf8]{inputenc}
\usepackage[T1]{fontenc}
\usepackage{aecompl}
\usepackage{ae}
\usepackage{textcomp}
\usepackage[english]{babel}
\usepackage{dsfont}
\usepackage{capt-of} 
\usepackage{mathtools}
\usepackage{etoolbox}

\usepackage{hyperref}
\hypersetup{colorlinks}
\usepackage{hypbmsec}
\usepackage{amsmath,amsthm,amsfonts}
\usepackage{bbm}
\usepackage{enumerate}
\usepackage[usenames]{color}

\theoremstyle{plain}
\newtheorem{theorem}{Theorem}[section]
\newtheorem{proposition}[theorem]{Proposition}
\newtheorem{corollary}[theorem]{Corollary}
\newtheorem{lemma}[theorem]{Lemma}
\newtheorem{remark}[theorem]{Remark}
\newtheorem{definition}[theorem]{Definition}

\newcommand{\E}{\mathbb{E}}
\renewcommand{\P}{\mathbb{P}}
\newcommand{\Z}{\mathbb{Z}}
\newcommand{\N}{\mathbb{N}}
\newcommand{\R}{\mathbb{R}}

\newcommand{\T}{\mathcal{T}}

\renewcommand{\d}{\mathrm{d}}

\newcommand{\gw}{\mathbf{\mathsf{GW}}}
\newcommand{\rgw}{\mathbf{\mathsf{RGW}}}
\newcommand{\rc}{\mathbf{\mathsf{R}}}

\usepackage{cmll}

\newcommand{\indicator}[1]{\mathds{1}_{#1}}

\renewcommand{\r}{\mathsf{r}}

\usepackage{fancyhdr}
\hypersetup{%
  pdftitle = {Rotor-routing on Galton-Watson trees},
  pdfauthor = {Wilfried Huss, Sebastian Müller, Ecaterina Sava-Huss},
  pdfkeywords = {rotor-router model, galton-watson tree}
}

\title{Rotor-routing on Galton-Watson trees}
\author {Wilfried Huss\footnote{Cornell University, USA; \texttt{huss@math.tugraz.at}}, Sebastian Müller\footnote{University Aix-Marseille, France; \texttt{sebastian.muller@univ-amu.fr}}
and Ecaterina Sava-Huss\footnote{Cornell University, USA; \texttt{sava-huss@cornell.edu}}}

\begin{document}

\AtEndEnvironment{thebibliography}{
\bibitem[LP15]{LP:book}
R. Lyons {\rm with} Y. Peres,
\newblock Probability on Trees and Networks.
\newblock \emph{Cambridge University Press}, (2015).
\newblock In preparation.  Current
  version available at \hfill\break
  \url{http://pages.iu.edu/\string~rdlyons/}.
}

\maketitle

\begin{abstract}
A rotor-router walk on a graph is a deterministic process, in which each vertex is endowed with a rotor
that points to one of the neighbors. A particle located at some vertex first rotates the rotor in a prescribed order, and then it is routed to 
the neighbor the rotor is now pointing at. In the current work we make a step toward in understanding the behavior of rotor-router walks on random trees. 
More precisely, we consider random i.i.d.~initial configurations of rotors on Galton-Watson trees $\T$, i.e.~on a family
tree arising from a Galton-Watson process, and give a classification in recurrence and transience for
rotor-router walks on these trees. 
 
\end{abstract}

\textbf{Keywords:} Galton-Watson trees, rotor-router walk, recurrence, transience,
return \\ probability.

\textbf{Mathematics Subject Classification:} 60J80; 05C81; 05C05.

\section{Introduction}

A rotor-router walk on a graph is a deterministic process in which the exits from each vertex follow a prescribed periodic sequence. For an overview and other properties, see the expository paper \cite{chip_rotor_2008}.
Rotor-router walks capture in many aspects the expected behavior of simple random walks, but with significantly reduced fluctuations compared to a typical random walk trajectory;
for more details see \cite{cooper_spencer_2006,
friedrich_levine, holroyd_propp}. 
However, this similarity breaks down when one looks at recurrence or transience of the walks, where the rotor-router walk may behave differently than the corresponding random walk.
When referring to rotor-router walks, we shall sometimes use only the shorter name \textit{rotor walks}.

A key result of Schramm states that for any choice of the initial rotor configuration, the rotor-router walk is in a certain sense no more transient than the random walk.
A proof of this result is presented in \cite[Theorem 10]{holroyd_propp}. The other direction is more sensitive, since it depends on the choice of the
initial configuration. There exist graphs where the simple random walk is transient, but the rotor-router walk is still recurrent.
We say that a rotor-router walk which started at the origin with initial rotor configuration $\r$ is \textit{recurrent} if it 
returns to the origin infinitely many times. Otherwise we say that the rotor-router walk is \textit{transient} (or the rotor configuration $\r$ is transient).
In \cite{angel_holroyd_2012} it is shown that for any $d\geq 1$ and any initial rotor mechanism (cyclic order in which the neighbors are served) on $\Z^d$, there exists a
recurrent rotor configuration (which send particles initially back toward the origin); on $\Z^2$ an example of an initial rotor configuration
for which the rotor-router walk is recurrent was given in \cite[Theorem 5]{holroyd_propp}. See also \cite{landau_levine_2009} for homogeneous trees 
and \cite{florescu_ganulgy_levine_peres_2014} for initial rotor configurations with all rotors aligned on $\mathbb{Z}^d$. 
On general trees, the issue of transience and recurrence was studied in detail in \cite{angel_holroyd_2011}. An extension 
for random initial configuration of rotors was made in \cite{huss_sava_transience_dircovers} on directed covers of graphs.

In this note we investigate the recurrence and transience properties of rotor-router walks with random initial rotor configuration $\rho$ on Galton-Watson trees.
The main result Theorem \ref{main_thm} gives a criterion for 
recurrence and transience of rotor-router walks. Moreover, if we run $n$ rotor walks starting from the root and record whether each walk returns to the root or escapes to infinity,
we show that in the transient regime  the relative density of escapes of the rotor-router walk
 equals almost surely the  return probability of the simple random walk on Galton-Watson trees. 
If $\T$ is the family tree of a Galton-Watson process, and $\gw$ the law of the family tree, then as a consequence of the main result we get the following.

\begin{theorem}
\label{main_corollary}
 Let $\rho$ be a uniformly distributed rotor configuration on a Galton-Watson tree $\T$ with mean offspring number $m$.
 Then for $\gw$-almost all trees $\T$, the rotor-router walk on $\T$ is recurrent if and only if $m\leq 2$. 
\end{theorem}

\section{Preliminaries}
\subsection{Galton-Watson trees}

Let $\left(Z_h\right)_{h\in\N_0}$ be a Galton-Watson process  with offspring distribution $\xi$ given by $p_k = \P[\xi = k]$ for $k\in\N_0$. Throughout the paper we assume $p_0=0$; this assumption is made for presentational reasons.
Informally, a Galton-Watson process is defined as follows: we start with one particle $Z_0=1$, which has $k$ children with probability $p_k$. Then each of these children also have children with the same offspring distribution, independently of each other and of their parent. This continues forever. 
Formally, let $\left(\xi^h_i\right)_{i,h\in\N}$ be i.i.d.~random variables with
the same distribution as $\xi$. 
Define $Z_h$ to be the size of the $h$-th generation, that is,
\begin{equation*}
Z_{h+1} = \sum_{i=1}^{Z_h} \xi_i^h,
\end{equation*}
and $Z_0 = 1$. Our assumption $p_0 = 0$ means that each vertex has at least one child, and the process survives almost surely. That is
we have that \textit{the mean offspring number $m=\E[\xi]\geq1$}.
Here and thereafter we denote by  $\E[\xi]=\sum_{k\geq 1} k \xi_{k}$ for a stochastic vector $\xi$.
Observe hereby that a supercritical Galton-Watson process conditioned on survival is a Galton-Watson process with $p_{0}=0$ plus some finite bushes. The existence of finite bushes does not influence recurrence and transience properties. 

We will not be interested only in the size $Z_h$ of the $h$-th generation, but also in the underlying family trees. Let $\T$ be the family tree of this Galton-Watson process, with vertex set
$V(\T) = \big\{x_i^h: h\in\N_0,\, 1\leq i \leq Z_h\big\}$. 
If $x_j^{h+1}$ is a descendant of $x_i^h$
there is an edge in the family tree between these two vertices. For ease of notation we will always
identify $\T$ with its vertex set. Denote by $\gw$ the law of the family trees of
our Galton-Watson process. For technical reasons we will add one additional vertex $s$
to the tree, which will act as the parent of the root vertex $o = x_0^1$. 
The vertex $s$ is called \textit{the sink} of the tree. 
We always consider $\T$ together with its natural planar embedding, that
is for each generation $h$ we draw the vertices $x_i^h$ from right to left, for $i=1,\ldots,Z_h$.

For each vertex $x\in \T$, denote by $\d_x$ the (random) number of children of $x$. Given the planar embedding described above, we denote by $x^{(k)}$, $k=0,\ldots,\d_x$,
the neighbors of $x$ in $\T$ in counterclockwise order beginning at the parent $x^{(0)}$ of $x$.

Throughout the paper we use the notation $\T$ for $\gw$-distributed random trees,
while $T$ will be used for fixed trees. Similarly we use $\r$ for fixed rotor configurations and $\rho$ for the random rotor configurations on $\T$.

\subsection{Rotor-router walks}

Let $T$ be an infinite tree with root (or origin) $o$. 
 A rotor configuration $\r$ on $T$
is a map $\r:T \to \N_0$ such that $\r(x) \in \{0,\ldots, \d_x\}$ for all vertices $x\in T$. For a given
rotor configuration $\r_0$ and a starting vertex $x_0\in T$, a rotor-router walk is a sequence of pairs
$\big\{(x_i,\r_i)\big\}_{i\geq 0}$ such that for all $i \geq 1$ we have the transition rule
\begin{equation*}
\r_{i+1}(x) = 
\begin{cases}
(\r_i(x_i) + 1) \mod (\d_{x_i}+1),  \quad &\text{if } x=x_i\\
\r_i(x), &\text{otherwise},
\end{cases}
\end{equation*}
and $x_{i+1} = x_i^{(\r_{i+1}(x_i))}$. Informally this means that a particle performing a rotor-router walk, when reaching the vertex $x$ first
increments the rotor at $x$ and then moves to the neighbour of $x$ the  rotor is now pointing at.

Depending on the initial rotor configuration $\r$, the rotor-router walk can exhibit one of the
following two behaviors: either the walk eventually returns to the origin
or the walk never returns to the origin and visits each vertex only finitely many times; see e.g. \cite[Lemma 6]{holroyd_propp}.
If any vertex would be visited infinitely often, then each of its neighbors must be visited infinitely often, and so the origin
itself would be visited infinitely often and we are in the \textit{recurrent case} of a rotor-router walk. In the second case,
when each vertex is visited only finitely many times, we are in the \textit{transient case}, and we say that the rotor-router walk
\textit{escapes to infinity}. Here, the rotor-router walk leaves behind a well defined limit rotor configuration.
In the transient case, we want also to quantify \textit{how transient} an initial configuration $\r$ for a rotor-router walk can be.
In order to do so, we start with $n$ rotor particles at the origin of $T$ and let each of them perform a rotor-router walk
until either it returns to the origin or it escapes to infinity. More formally, let $\r_1 = \r$ be a fixed initial rotor configuration, 
and denote by $\r_n$ the initial rotor
configuration of the $n$-th particle. For all $n\geq 1$, run the $n$-th
rotor-router particle until it returns to $o$ for the first time. If this occurs after a finite number of
steps, let $\r_{n+1}$ be the rotor configuration left behind by this particle. In case the $n$-th particle never
returns to $o$ we define $\r_{n+1}$ to be the limit configuration created by the $n$-th particle escaping to infinity.
For $k=1,\ldots,n$, let 
\begin{equation*}
e_k = \begin{cases}
  1,\quad&\text{if the $n$-th particle escapes to infinity} \\
  0,     &\text{otherwise},
\end{cases}
\end{equation*}
and let $E_n(T,\r) = \sum_{k=1}^n e_k$ count the number of escapes of the first $n$ walks.

The next result, due to Schramm states that a rotor-router walk is no more transient than a random walk.
A proof of this result can be found in \cite[Theorem 10]{holroyd_propp}.

\begin{theorem}\label{thm:schramm_thm}
For any locally finite graph $G$, any starting vertex, any cyclic order of neighbors and any initial rotor configuration $\r$
\begin{equation*}
\limsup_{n\to\infty}\frac{E_n(G,\r)}{n} \leq \gamma(G),
\end{equation*}
where $\gamma(G)$ represents the probability that the simple random walk on the graph $G$ never returns to the starting vertex.
\end{theorem}
This result suggests the following question: when does $\lim_{n}\frac{E_n(G,\r)}{n}$ exist, and when is it equal to $\gamma(G)$?
For general  state spaces $G$, the answer depends both on $G$ and on the initial configuration of rotors.
We give an answer to this question for rotor-router walks on Galton-Watson trees $\T$, that is we prove that the limit 
$\lim_{n}\frac{E_n(\T,\rho)}{n}$ exists and equals $\gamma(\T)$ almost surely,
for a random initial rotor configuration $\rho$.
This is the statement of the Theorem \ref{main_thm}.

\paragraph{Abelian property.} Rotor-router walks possess a number of interesting properties, one of which will be used several times during the current work.
This is the \textit{Abelian property}, e.g. \cite[Proposition 4.1]{diaconis_fulton_1991}, which allows several rotor-router walks to walk ``simultaneously'' instead of ``successively''. To make this precise we follow \cite{angel_holroyd_2011}. 

Let $T$ be a finite tree  with root $o$ connected to an additional vertex $s$ (the parent of the root). Furthermore, let $S$, with $s\notin S$, be a non-empty set of the leaves of $T$, and $\r$ be an initial rotor configuration on $T$. Suppose that at each vertex there is some nonnegative number of particles.
At each step of the process we choose a vertex $v\notin S\cup\{s\}$ at which there is at least one particle (if such a vertex exists) and perform one step of a rotor-router walk step with this particle. Such an operation is called a \textit{legal move}. A \textit{legal sequence} is a sequence of legal moves.
A \textit{complete legal sequence} is a legal sequence, such that at the end
of the sequence no further legal moves are possible.

\begin{lemma}[Lemma 24,\cite{angel_holroyd_2011}]
Consider the setting of the last paragraph.
If we start with $n$ particles at $o$, and perform any complete legal sequence, then the process terminates in a finite number of steps, and the number of particles in $S$ is exactly $E_n(T,\r)$.
\end{lemma}

\section{Random initial rotor configuration}
We construct random initial rotor configurations on Galton-Watson trees $\T$.
In order to do this, for each $k\geq 0$ we choose a probability
distribution $\mathcal{Q}_k$ supported on $\{0,\ldots,k\}$. That is, we have the sequence of distributions $(\mathcal{Q}_k)_{k\in\N_0}$, where
\begin{equation*}
\mathcal{Q}_k = \big(q_{k,j}\big)_{0\leq j\leq k}
\end{equation*}
with $q_{k,j}\geq 0$ and $\sum_{j=0}^k q_{k,j} = 1$. Let $\mathcal{Q}$ be the infinite lower
triangular matrix having $\mathcal{Q}_k$ as row vectors, i.e.:
\begin{equation*}
\mathcal{Q} = 
\begin{pmatrix}
q_{00} & 0      & 0      & 0      & \hdots \\
q_{10} & q_{11} & 0      & 0      & \hdots \\
q_{20} & q_{21} & q_{22} & 0      & \hdots \\
q_{30} & q_{31} & q_{32} & q_{33} & \hdots \\
\vdots & \vdots & \vdots & \vdots & \ddots \\
\end{pmatrix}.
\end{equation*}

\begin{definition}
A random rotor configuration $\rho$ on a tree $T$ is $\mathcal{Q}$-distributed,
if  for each $x\in T$, the rotor $\rho(x)$ is a random variable with the following properties:
\begin{enumerate}
\item $\rho(x)$ is $\mathcal{Q}_{\d_x}$ distributed, i.e.,
$\P[\rho(x) = \d_x - l \,|\, \d_v = k] = q_{k,l}$, with $l=0,\ldots \d_x$,
\item $\rho(x)$ and $\rho(y)$ are independent if $x\not=y$, with $x,y\in T$.
\end{enumerate}
We write $\rc_{T}$ for the corresponding probability measure.
\end{definition}
Then $\rgw=\rc_{\T}\times \gw$  represents 
the probability measure given by choosing a tree $\T$ according to the $\gw$ measure, and then independently
choosing a rotor configuration $\rho$ on $\T$ according to $\rc_{\T}$. 
Recall that to the root $o\in\T$, we have added an additional vertex $s$, which should be considered as the parent of the root.

We are now ready to state our main result.
\begin{theorem}
\label{main_thm}
Let $\rho$ be a random $\mathcal{Q}$-distributed rotor configuration on a Galton-Watson tree $\T$ with offspring distribution $\xi$, and let
$\nu = \xi\cdot\mathcal{Q}$. Then we have  for $\rgw$-almost all~$\T$ and $\rho$:
\begin{enumerate}[(a)]
\item \label{main_thm_a}$\displaystyle E_n(\T,\rho) = 0$ for all $n\geq 1$, if $\E[\nu] \leq 1$,
\item \label{main_thm_b}$\displaystyle\lim_{n\to\infty} \frac{E_n(\T,\rho)}{n} = \gamma(\mathcal{T})$, if $\E[\nu] > 1$,
\end{enumerate}
where $\gamma(\T)$ represents the probability that simple random walk started at the root  of $\T$  never returns  to $s$. 
\end{theorem}
Theorem \ref{main_thm} is a generalization of \cite[Theorem 6]{angel_holroyd_2011}
which holds for regular trees. See also \cite[Theorem 3.5]{huss_sava_transience_dircovers}
for rotor-router walks on periodic trees.

Theorem \ref{main_corollary} is now a corollary of Theorem \ref{main_thm} in the case of uniformly
distributed rotors. More precisely, if $\mathcal{Q}_k$ is the uniform distribution on $\{0,\ldots,k\}$
for all $k$, we get a particularly simple recurrence condition involving only the mean offspring
number $m=\E[\xi]$.

\begin{proof}[Proof of Theorem \ref{main_corollary}]
$\displaystyle\quad
\E[\nu] = \sum_{l=0}^\infty \sum_{k=l}^\infty \frac{1}{k+1} p_k = \frac{1}{2} \sum_{k=0}^\infty k p_k = \frac{1}{2} \E[\xi]$.
\end{proof}

\begin{remark}
Notice here the difference between simple random walk on Galton-Watson trees which is transient for mean offspring number $m\in(1,2]$, while the rotor-router walk with uniformly distributed initial rotors is recurrent.
\end{remark}

\subsection{Recurrent part}

\begin{proof}[Proof of Theorem \ref{main_thm}\eqref{main_thm_a}]
Let $\rho$ be a random $\mathcal{Q}$-distributed rotor configuration on a Galton-Watson tree $\T$. Recall that for a vertex $x\in\T$ we denote by $x^{(k)}$, $k=1,\ldots,
\d_x$, the children of $x$. We call a child $x^{(k)}$ \emph{good} if $\rho(x) < k$. This means that the rotor walk at $x$ will visit the good
children before visiting the parent $x^{(0)}$. Since $\rho$ is $\mathcal{Q}$-distributed,
\begin{equation*}
\P[x \text{ has $l$ good children} \,|\, \d_x = k] = q_{k,l}.
\end{equation*}
Therefore the distribution of the number of good children of a vertex $x$ in $\T$ is given by
\begin{equation*}
\P[x \text{ has $l$ good children}] = \sum_{k = l}^{\infty} p_k q_{k,l},
\end{equation*}
which is the $l^{\text{th}}$ component of the vector $\nu = \xi\cdot\mathcal{Q}$.
Thus, for each vertex $x$ the set of descendants of $x$ which are connected to $x$ by a
path consisting of only good children forms a Galton-Watson tree with offspring distribution
$\nu$. By \cite[Proposition 8]{angel_holroyd_2011} the rotor walk can only escape to
infinity along paths that consist exclusively of good children. By assumption we have $\E[\nu] \leq 1$,
hence subtrees consisting of only good children die out almost surely. This implies that there are no escapes to infinity.
\end{proof}

\subsection{The frontier process}\label{sec:frontier}
To prove the transient part of Theorem \ref{main_thm}, we will use the \emph{frontier
process} introduced in \cite{huss_sava_errata_dircovers}. For sake of completeness we state
the definition of the process here.

Fix an infinite tree $T$ with root $o$ and without leaves and a rotor configuration $\r$ on $T$. As
before we attach an additional vertex $s$ to the root $o$ of the tree. Consider the
following process which generates a sequence
$F_\r(n)$ of subsets of vertices of the tree $T$. $F_{\r}(n)$ is constructed by a
rotor-router process consisting
of $n$ rotor-router walks starting at the root $o$, such that each vertex of $F_{\r}(n)$ contains
exactly one particle.  In the first step put a particle at the root $o$ and set $F_{\r}(1) = \{o\}$.
Inductively given $F_{\r}(n)$ and the rotor configuration that was created by the previous step, we
construct the next set $F_{\r}(n+1)$ using the following rotor-router procedure. Perform rotor-router
walk with a particle starting at the root $o$, until one of the following stopping conditions
occurs:
\begin{enumerate}[(a)]
\item\label{frontier_a} The particle reaches $s$. Then set $F_{\r}(n+1) = F_{\r}(n)$.
\item\label{frontier_b} The particle reaches a vertex $x$, which has never been visited before. Then set \\ $F_{\r}(n+1) = F_{\r}(n) \cup \{x\}$.
\item\label{frontier_c} The particle reaches an element $y\in F_{\r}(n)$. We delete $y$ from $F_{\r}(n)$, i.e.,
set $F'(n) = F_{\r}(n)\setminus\{y\}$. Now there are two particles at $y$, both of which are
restarted until stopping condition \eqref{frontier_a}, \eqref{frontier_b} or \eqref{frontier_c} for
the set $F'(n)$ applies to them. Note that since we are on a tree at least one particle will stop at a child of $y$ after one step, due
to halting condition \eqref{frontier_b}.
\end{enumerate}

We will call the set $F_{\r}(n)$ the \emph{frontier} of $n$ particles.
Basic properties of this process can be found in \cite{huss_sava_errata_dircovers}. Note that the frontier process $F_{\r}(n)$ depends on the underlying tree $T$.
Following \cite{huss_sava_errata_dircovers}, let us introduce 
\begin{equation}\label{eq:max_height}
M(n) = \max_{\r\in\mathcal{R}(T)} \max_{x\in F_\r(n)}|x|, 
\end{equation}
where $|x|$ is the distance of $x$ to the root $o$ and $\mathcal{R}(T)$ is
the set of all rotor configurations on $T$.
$M(n)$ represents the maximal height of all possible frontiers of $n$ particles. 
In order to get an upper bound for $M(n)$,
we will use the \textit{anchored expansion constant}.

Let $G$ be an infinite graph with vertex set $V=V(G)$ and edge set $E=E(G)$. For $A\subset V(G)$ we denote by $|A|$ the cardinality of $A$ and by $\partial A$
the \textit{vertex boundary} of $A$, that is, the set of vertices in $V(G)\setminus A$ that have one neighbor in $A$.
We say that the set $A$ is connected if the induced subgraph on $A$ is connected. Fix the root $o\in V(G)$.
The anchored expansion constant of $G$ is
\begin{equation}\label{eq:anch_exp}
 \iota^*_{E}(G)=\liminf_{n\to\infty}\bigg\{\frac{|\partial A|}{|A|}:o\in A\subset V(G),\ A
 \text { is connected}, \ n\leq |A|<\infty\bigg\}.
\end{equation}

For the remainder of this section, we assume that the tree $T$ has positive anchored expansion constant $\iota^{*}_{E}(T)>0$.

The next result is a generalization of \cite[Lemma 1.5]{huss_sava_errata_dircovers}.

\begin{lemma}\label{lem:mn_bound}
Assume $\iota^{*}_{E}(T)>0$. Then there exists $c=c(T) < 1$ such that $M(n) < c n$, for $n$ large enough.
\end{lemma}
\begin{proof}
Consider the frontier process $F_{\r}(n)$ on $T$.
Let $x$ be an element of $F_{\r}(n)$ with maximal distance $M = |x|$ to the root $o$. Denote by
$p = (o = x_0 , x_1, \ldots, x_M = x)$ the shortest path between $o$ and $x$. Since $F_{\r}(1) = \{o\}$
and by the iterative construction of $F_{\r}(n)$, there exist $1 = n_0 < n_1 < \cdots < n_M = n$, such that
$x_i \in F_{\r}(n_i)$ for all $i\in 0,\ldots,M$. 

We want to find a lower bound for $n_{i+2} - n_i$, that is, for the number of steps needed to replace
$x_i$ by $x_{i+2}$ in the frontier. At time $n_i$, the vertex $x_i$ is added to the frontier.
The next time after $n_i$ that a particle visits $x_i$ halting condition \eqref{frontier_c} occurs,
thus the rotor at $x_i$ is incremented two times. As long as not all children of $x_i$
are part of the frontier, every particle can visit $x_i$ at most once, since it either stops immediately
at a child of $x_i$ on stopping condition \eqref{frontier_b} or is returned to the ancestor of $x_i$. This means that
at subsequent visits the rotor at $x_i$ is incremented exactly once. In order for $x_{i+2}$ to be added
to the frontier, the rotor at $x_i$ has to point at direction $x_{i+1}$ twice. Thus replacing $x_i$ with
$x_{i+2}$ in the frontier, needs at least $\d_{x_i}+1$ particles which visit $x_i$.
Hence, $n_{i+2} - n_i \geq 1+\d_{x_i}$, for $i=0,\ldots,M-2$.
Denote by $\tilde{p} = (x_0 , x_1, \ldots, x_{M-3}, x_{M-2})$.
By assumption, we have $\iota^*_E(T) > 0$. Thus, for $M$ big
enough there exists a constant $\kappa > 0$, such that
$|\partial \tilde{p}| \geq \kappa |\tilde{p}|$.
Since $\tilde{p}$ is a path of a tree we get
\begin{equation*}
|\partial\tilde{p}| + |\tilde{p}| = 1 + \sum_{i=0}^{M-2} \d_{x_i}.
\end{equation*}
We have therefore
\begin{equation}
\label{eq:n_i_lower}
\begin{aligned}
\sum_{i=0}^{M-2} n_{i+2} - n_i &\geq \sum_{i=0}^{M-2} \big(1+\d_{x_i}\big) =
(M-1) + |\tilde{p}| + |\partial{\tilde{p}}| - 1 \\
&\geq (M-1)(\kappa + 2) - 1.
\end{aligned}
\end{equation}
On the other hand
\begin{equation}
\label{eq:n_i_upper}
\sum_{i=0}^{M-2} n_{i+2} - n_i = n_M + n_{M-1} - n_1 - n_0 < 2 n.
\end{equation}
Combining \eqref{eq:n_i_lower} and \eqref{eq:n_i_upper} gives
$M \leq \frac{2}{\kappa+2} n + \frac{1}{\kappa+2} + 1$, which proves the claim.
\end{proof}
We aim now at getting a lower bound for the size of $F_{\r}(n)$. Following \cite{huss_sava_errata_dircovers}, we first estimate the number of particles stopped at $s$ during the
formation of the frontier $F_{\r}(n)$. This is accomplished using Theorem 1 from \cite{holroyd_propp}. For a tree $T$, a rotor configuration $\r$, and $n\geq 1$ define
\begin{equation}\label{eq:fill_holes}
 \ell(n)=\{x\in T: |x|=M(n) \text{ and the path from $o$ to $x$ contains no vertex of } F_{\r}(n)\},
\end{equation}
with $M(n)$ defined as in \eqref{eq:max_height}.
By construction, $T\setminus F_{\r}(n)$ does not need to have a finite component.
In order to rectify this, we add additional vertices $\ell(n)$ introduced in \eqref{eq:fill_holes} on the maximal level $M(n)$.
All these additional vertices were not touched by a rotor particle during the formation of $F_{\r}(n)$. Let now 
\begin{equation}\label{eq:frontier_sink}
S=F_{\r}(n)\cup \ell(n)
\end{equation}
be the sink determined by the frontier process $F_{\r}(n)$. 
Then the component of $T\setminus S$ containing the root is finite. We denote this component by $T^S$, and we refer to $T^S$ as the truncation of $T$
at the sink $S$.

Let now $(X_t)$ be the simple random walk on $T$.
Furthermore, let $\tau_s=\inf\{t\geq 0: X_t= s\}$ and $\tau_S=\inf\{t\geq 0: X_t\in S\}$ be the first hitting time of $s$ and $S$ respectively.
Consider now the hitting probability
\begin{equation}\label{eq:hit_pb}
 h(x)=h_s^{S}(x)=\P_x[\tau_s<\tau_{S}],\quad x\in T^S,
\end{equation}
that is, the probability to hit $s$ before $S$, when the random walk starts in $x$. 
We have $h(s)=1$ and $h(x)=0$, for all $x\in S$. For $x\in T\setminus T^S$, we set $h(x)=0$.  

Start now $n$ rotor particles at the root $o$, and stop them when they either reach $s$ or $S$. By the Abelian property of rotor-router walks and by the construction of the frontier process $F_{\r}(n)$ we will have exactly one rotor particle at each vertex of $F_{\r}(n)$, no particles at $\ell(n)$, and the rest of the particles
are at $s$. In order to estimate the proportion of rotor particles stopped at $s$ we use Theorem 1 from \cite{holroyd_propp}, which we state here adapted to our case.
\begin{theorem}[Theorem 1, \cite{holroyd_propp}]
Consider the sink $S$ as above, and let $(X_t)$ be the simple random walk on $T$.
Let $E$ be the edge set of $T$ and suppose that the quantity 
\begin{equation}\label{eq:const_k}
K=1+\sum_{(x,y)\in E}|h(x)-h(y)|
\end{equation}
is finite. If we start $n$ rotor particles at the root $o$,
then
\begin{equation}\label{eq:proportion_estimate}
 \Big|h(o)-\frac{n_{s}}{n}\Big|\leq \frac{K}{n},
\end{equation}
where $n_{s}$ represents the number of particles stopped at $s$.
\end{theorem}

We make use of the following result, whose proof can be found in \cite[Lemma 1.7]{huss_sava_errata_dircovers}.
\begin{lemma}\label{lem:k_bound}
Let $K$ be the constant defined in \eqref{eq:const_k}. Then $ K= 1+\big(M(n)+1\big) \big(1-h(o)\big)$.
\end{lemma}

\begin{corollary}\label{cor:down_sink}
Assume $\iota^{*}_{E}(T)>0$.
There exists $\kappa=\kappa(T)\in (0,1)$, such that $\#F_{\r}(n)>\kappa n$, for $n$ large enough.
\end{corollary}
\begin{proof}
 From \eqref{eq:proportion_estimate}, we have $\frac{n_{s}}{n}\leq \frac{K}{n} + h(o)$.
 Putting together Lemma \ref{lem:mn_bound} and \ref{lem:k_bound}, we obtain $K< 1+(cn+1) (1-h(o))$. Then we have
 \begin{equation*}
 n_s\leq n\left(h(o)(1-c)+c+\frac{2}{n}\right).
 \end{equation*}
 Since $h(o)(1-c)+c<1$ we can choose a $\kappa' < 1$ such that $n_s < \kappa' n$,
 for $n$ large enough. Since $\#F_{\r}(n)=n-n_s$, the claim follows by putting $\kappa = 1-\kappa'$.
\end{proof}

We shall use the following result of Chen and Peres \cite{chen_peres}.
\begin{theorem}[Theorem 6.36, \cite{LP:book}]
\label{gw_anchored_expansion}
For a supercritical Galton-Watson tree $\T$, given non-extinction we have $\iota^{*}_{E}(\T)>0$, $\gw$-almost surely. 
\end{theorem}

Therefore,  if we perform the frontier process $F_{\rho}(n)$ on a Galton-Watson tree $\T$ with random rotor configuration $\rho$, we get by Corollary \ref{cor:down_sink}
that there exists a positive ($\rgw$-almost surely) random variable ${\kappa}$ such that for $n$ large enough $\#F_{\rho}(n)>{\kappa} n$, $\rgw$-almost surely.

\subsection{Transient part}
In this section we will prove the transient part of Theorem \ref{main_thm}.
Let $T$ be a tree with root $o$, and rotor configuration $\r$. Denote by
$T_1,\ldots,T_{\d_o}$ the principal branches of $T$, and by $\r_j$ the
restriction of $\r$ to $T_j$.
Write
\begin{equation*}
 l(T,\r)=\liminf_{n\to\infty}\dfrac{E_n(T,\r)}{n} \quad \text{ and }\quad 
 l_j(T,\r)=\liminf_{n\to\infty}\dfrac{E_n(T_j,\r_j)}{n}, \quad j=1,\ldots,\d_o.
\end{equation*}
From \cite[Lemma 25]{angel_holroyd_2011}, we have that for any tree $T$
and rotor configuration $\r$
\begin{equation}\label{eq:l_iteration2}
 l(T,\r)\geq 1-\dfrac{1}{1+\sum_{j=1}^{\d_o}l_j(T,\r)}.
\end{equation}

For a Galton-Watson tree $\T$, denote like above $\T_1,\ldots,\T_{\d_o}$ the $\d_o$ principal branches of $\T$,
where $\d_o$ represents the random degree of the root $o$ of $\T$, with distribution $\xi$ (the offspring distribution of the
Galton-Watson process). All $l_j(\T,\rho)$, $j=1,\ldots,\d_o$ are i.i.d.
under $\rgw$, with the same law as $l(\T,\rho)$, but $l(\T,\rho)$ is not independent of the $l_j(\T,\rho)$.
Even more, for $\rgw$-almost every tree $\T$ and configuration $\rho$, \eqref{eq:l_iteration2} holds.
We first show that under the conditions of Theorem \ref{main_thm}\eqref{main_thm_b} the random variable
$l(\T,\rho)$ is greater than zero with positive probability.

\begin{proposition}\label{cor:l_delta}
Let $\rho$ be a random $\mathcal{Q}$-distributed rotor configuration on a Galton-Watson tree $\T$ with
offspring distribution $\xi$, and let $\nu = \xi\cdot\mathcal{Q}$. Suppose $\E[\nu]>1$. Then $l(\T,\rho) > 0$, $\rgw$-almost surely.
\end{proposition}
\begin{proof}
The proof follows in essence the proof of  \cite[Theorem 6(i)]{angel_holroyd_2011} but some additional arguments are needed to treat Galton-Watson trees instead of homogeneous trees. For sake of completeness
we give the full argument here.

Recall the definition of the embedded Galton-Watson process of good children defined in the proof of Theorem \ref{main_thm}\eqref{main_thm_a}.  By assumption we have now that $\E[\nu]>1$, therefore the above mentioned branching process survives with some positive probability $\mathsf{p}$. A \textit{live path} is an infinite path, say $(x_1, x_2,\ldots)$, such that $x_{i+1}$ is a good child of $x_i$, for all $i\geq 1$. The important property of a live path is, that a particle starting in $x_{1}$ will escape to infinity without returning to $x_1$. 
The assumption $\E[\nu]>1$ implies that for all $x\in\T$ with probability $\mathsf{p}$ there exists a live path starting at $x$.

In the following a realization of $\rgw$ is constructed in the following way.  
We start with $n$ particles at the root and build the frontier process $F_{\rho}(n)$ according to the algorithm described in Section \ref{sec:frontier}. Now the first generation of the Galton-Watson tree is given by a realization of $\xi$ and a rotor configuration is attributed to the origin $o$
according to the distribution $\mathcal{Q}_{\d_o}$. Inductively, whenever a vertex 
sends out a particle for the first time, we construct randomly (and independently of everything before) the children and rotor configuration of this vertex. This is repeated until the frontier process $F_{\rho}(n)$ is constructed. Then,  each vertex that was not yet visited serves as the root of an independent copy of $\rgw$. 
 The law of the behavior of the particles is the same as under $\rgw$. 
Denote by $X$ the set of vertices in $F_{\rho}(n)$, for which there is a live path starting at $x$.
 We shall first prove that
\begin{equation}\label{eq:ldev}
E_n(\T,\rho)\geq \# X.
\end{equation}
From \cite[Lemmas 18,19]{holroyd_propp}, it suffices to prove
\eqref{eq:ldev} for $\T^H=\{x\in\T:|x|\leq H\}$, with $H>M(n)$, i.e.,
\begin{equation}\label{eq:esc_truncated}
 E_n(\T^H,S^H,\rho^H)\geq \# X.
\end{equation}
Here, $E_n(\T^H,S^H,\rho^H)$ represents the number of particles that stop at $S^H=\{x\in\T:|x|=H\}$ when we start $n$ rotor-router walks at the root $o$ of
$\T$ and rotor configuration $\rho^H$ (the restriction of $\rho$ on $\T^H$).
In $\T^S$, where recall that $\T^S$ represents the tree $\T$ truncated at the frontier $S$, start $n$ particles at $o$, and stop them when they either reach $S$ or return to $s$. The vertices
at distance greater than $M(n)$ were never visited, and the rotors there are in their initial random configuration.
Now for every vertex $x$ in $X$ restart one particle. 
Since there is a live path at $x$ the particle will reach the level $H$ without leaving the cone of $x$, at which point the particle is stopped again.
Hence if we restart all particles which are located in $F_{\rho}(n)$ at least $\# X$ of them will reach
level $H$ before returning to the root. Because of the Abelian property of rotor-router walks,
\eqref{eq:esc_truncated} follows, therefore also \eqref{eq:ldev}.

The random variable $\# X$ can be written as a sum of i.i.d. Bernoulli variables
in the following way. Consider the events
\begin{equation*}
A_x = \big[\text{The tree of good children with root $x$ is infinite}\big],
\end{equation*}
for all $x\in F_{\rho}(n)$. Then for all $x\in F_{\rho}(n)$, by construction of the
frontier the rotor configuration on the subtree rooted at $x$ is unchanged.
Thus the tree of good children rooted at $x$ is a Galton-Watson tree with offspring
distribution $\nu$. By assumption $\E[\nu]>1$, therefore $\P[A_x] = \mathsf{p}$.
Moreover the event $A_x$ depends only on the subtree rooted at $x$, hence the
events $(A_x)_{x\in F_{\rho}(n)}$ are independent. Let 
$F_\rho(n) = \big\{y_1,y_2,\ldots,y_{\#F_{\rho}(n)}\big\}$ and let $(Y'_i)_{i\geq 1}$ be a sequence of i.i.d. Bernoulli($\mathsf{p}$) random
variables, which are independent of $\T$ and $\rho$. Define
\begin{equation*}
Y_i = \begin{cases}
\indicator{A_{y_i}} & \text{ for } i \leq \# F_{\rho}(n), \\
Y'_i         & \text{ for } i > \# F_{\rho}(n).
\end{cases}
\end{equation*} 
Then we have that the $(Y_i)_{i\geq 1}$ are i.i.d. Bernoulli($\mathsf{p}$), and
\begin{equation*}
\# X = \sum_{i=1}^{\#F_{\rho}(n)} Y_i.
\end{equation*}
By Corollary \ref{cor:down_sink} and Theorem \ref{gw_anchored_expansion} there exists a positive random variable $\kappa$ and a constant $n_0$ such that 
\begin{align*}
\P[\#F_{\rho}(n)>\kappa n, \text{ for all } n\geq n_0] = 1.
\end{align*}
Since $\kappa$ is $\rgw$-a.s.~positive, for all $\epsilon>0$ we can choose a constant $\alpha(\epsilon) > 0$ such that
$\P[\kappa > \alpha(\epsilon) ] \geq 1-\epsilon$. 
Consider the two events 
\begin{align*}
B_{\kappa} = [\#F_{\rho}(n)>\kappa n, \text{ for all } n\geq n_0]
 \quad\text{and}\quad
B_{\alpha(\epsilon) } = [\#F_{\rho}(n)>\alpha(\epsilon)  n, \text{ for all } n\geq n_0].
\end{align*}
By the choice of $\alpha(\epsilon) $ we have
$\P[B_{\alpha(\epsilon)} ] \geq 1-\epsilon$. For any positive constant $\delta$ we have 
\begin{align*}
\begin{aligned}
\P\big[E_n(\T,\rho) <\delta n, B_{\alpha(\epsilon)} \big]  &\leq \P\big[\#X< \delta n, B_{\alpha(\epsilon)} \big]
= \P\left[\sum_{i=1}^{\# F_{\rho}(n)} Y_i < \delta n, B_{\alpha(\epsilon)} \right] \\
&\leq \P\left[\sum_{i=1}^{\alpha(\epsilon)  n} Y_i < \delta n\right] = \P\big[\widetilde{Y} < \delta n\big],
\end{aligned}
\end{align*}
where $\widetilde{Y}=\sum_{i=1}^{\alpha(\epsilon)  n}Y_i$ and $\E[\widetilde{Y}] = \alpha(\epsilon) \mathsf{p}n$.
Using the standard Chernoff bound
\begin{align*}
\P\big[E_n(\T,\rho) <\delta n, B_{\alpha(\epsilon)} \big] &\leq \P\left[\widetilde{Y} < \frac{\delta}{\alpha(\epsilon) \mathsf{p}}\E[\widetilde{Y}]\right] \leq
\exp\left\{-\frac{1}{2}\big(1-\tfrac{\delta}{\alpha(\epsilon) \mathsf{p}}\big)^2 \E[\widetilde{Y}]\right\} \\
&\leq \exp\left\{-\frac{1}{2}\big(1-\tfrac{\delta}{\alpha(\epsilon) \mathsf{p}}\big)^2 \alpha(\epsilon) \mathsf{p}n\right\}.
\end{align*}
Thus for $\delta_{\epsilon} < \alpha(\epsilon) \mathsf{p}$ there exists a constant $c_{\epsilon}>0$, such that, 
$\P\big[E_n(\T,\rho) <\delta_{\epsilon} n, B_{\alpha(\epsilon)} \big] < e^{-c_{\epsilon}n}$, which is summable. Hence using the Borel-Cantelli Lemma
\begin{align}\label{eq:limsup0}
\P\bigg[\limsup_{n\to\infty}\Big((E_n(\T,\rho) <\delta_{\epsilon} n) \cap B_{\alpha(\epsilon)}  \Big)\bigg] = 0.
\end{align}
Since the event $B_{\alpha(\epsilon) }$ does not depend on $n$, we have
\begin{equation*}
\limsup_{n\to\infty}\Big((E_n(\T,\rho) <\delta_{\epsilon} n) \cap B_{\alpha(\epsilon)}  \Big) =
\bigg(\limsup_{n\to\infty}\big(E_n(\T,\rho) <\delta_{\epsilon} n\big)\bigg) \cap B_{\alpha(\epsilon)}.
\end{equation*}
Thus taking the complement in \eqref{eq:limsup0} and applying the union bound gives

\begin{equation*}
\P\left[\liminf_{n\to\infty}\frac{E_n(\T,\rho)}{n} \geq \delta_{\epsilon}\right] \geq 1-\epsilon.
\end{equation*}
Since $\epsilon$ is arbitrary, the claim follows.
\end{proof}

Recall that $\gamma(\T)$ is the probability that simple random walk started at the root $o$  of $\T$  will never visit $s$, the parent of the root $o$.
The next step is to find the law of $\gamma(\T)$. From \cite[Equation (4.1)]{lyons_pemantle_peres_1997_2},
we have $\gw$-almost surely 
\begin{equation}\label{gw_escape_prob}
 \gamma(\T) = 1 - \frac{1}{1+\sum_{j=1}^{\d_o} \gamma(\T_j)},
\end{equation}
where $\d_o=\d_o(\T)$ represents the (random) degree of the root $o$ of $\T$, which has the same distribution as the offspring distribution $\xi$.

Denote by $F_{\gamma}$ the cumulative distribution function (c.d.f.)~of $\gamma(\T)$ and by $F_l$ the c.d.f. of $l(\T,\rho)$. We always assume that the two random variables $\gamma(\T)$ and $l(\T,\rho)$ are defined on the same probability space with probability measure $\rgw$.
Next, we want to prove that the random variable $l(\T,\rho)$ stochastically dominates $\gamma(\T)$, i.e.~$F_{l}(t)\leq F_{\gamma}(t)~\forall t\in\R$.

The recursive structure of Galton-Watson trees with offspring distribution $\xi$ and $\P[\d_o=k]=\P[\xi=k]=p_k$
gives that $F_{\gamma}$ satisfies
\begin{equation}\label{eq:cdf_ret_prob}
 F_{\gamma}(t)=
 \begin{cases}
  0, & \text{ if } t\leq 0\\
  \sum_{k=0}^{\infty}p_k F^{\star k}_{\gamma}\big(\frac{t}{1-t}\big), & \text{ if } 0 < t < 1\\
  1, & \text{ if } t \geq 1,
 \end{cases}
\end{equation}
where $F^{\star k}$ represents the $k$-th convolution power of $F$. Moreover, for $F_{l}$ we have by \eqref{eq:l_iteration2}
\begin{equation}\label{eq:cdf_l_inf}
 F_{l}(t)\leq \sum_{k=0}^{\infty}p_k F^{\star k}_{l}\Big(\frac{t}{1-t}\Big), \quad \text{ if } t\in (0,1).
\end{equation}
We shall use \cite[Theorem 4.1]{lyons_pemantle_peres_1997_2}, which we state here
in the version found in \cite{LP:book}.
\begin{theorem}[Theorem 16.32, \cite{LP:book}]\label{thm:operator_cdf}
The functional equation \eqref{eq:cdf_ret_prob} has a unique solution, $F_{\gamma}$. Define the operator on c.d.f.'s
\begin{equation*}
 \mathcal{K}:F \mapsto \sum_{k=0}^{\infty}p_k F^{\star k}\Big(\frac{t}{1-t}\Big),\quad
 \text{ with } t\in (0,1).
\end{equation*}
For any initial c.d.f.~$F$ with $F(0)=0$ and $F(1)=1$, we have weak convergence under iteration to $F_{\gamma}$:
\begin{equation*}
 \lim_{n\to\infty}\mathcal{K}^n(F)=F_{\gamma}.
\end{equation*} 
\end{theorem}
\begin{remark}\label{rem:ineq_oper_k}
Let $F_{1}$ and $F_{2}$ be two c.d.f.~with $F_1\leq F_2$, then
\begin{equation*}
\mathcal{K}^n(F_1)\leq \mathcal{K}^n(F_2) \quad \forall n\in \N. 
\end{equation*} This fact can be seen directly or can be found in the proof of \cite[Theorem 4.1]{lyons_pemantle_peres_1997_2}.
\end{remark}
\begin{lemma}\label{lem:k_fl}
For all $n\geq 1$, we have that $ F_l \leq \mathcal{K}^n(F_l)$.
\end{lemma}
\begin{proof}
 For $n=1$, we have $F_l \leq \mathcal{K}(F_l)$ which holds by \eqref{eq:cdf_l_inf}
 and by the definition of the operator $\mathcal{K}$ in Theorem \ref{thm:operator_cdf}.
 For each $n> 1$, using Remark \ref{rem:ineq_oper_k} for the c.d.f.'s $F_l \leq \mathcal{K}(F_l)$
 we get $\mathcal{K}^n(F_l)\leq \mathcal{K}^{n+1}(F_l)$. The claim follows then easily by induction.
\end{proof}

\begin{lemma}\label{lem:l_dom_stoc_e}
Suppose $\E[\nu]>1$. We have that $F_{l}\leq F_{\gamma}$.
\end{lemma}
\begin{proof}
By Proposition \ref{cor:l_delta}, $F_l(0)=0$. Hence, applying Theorem \ref{thm:operator_cdf} to $F_l$ gives 
\begin{equation*}
\lim_{n\to\infty}\mathcal{K}^n(F_l)=F_{\gamma}.
\end{equation*}
The claim follows now by Lemma \ref{lem:k_fl}.
\end{proof}

We are finally able to prove the transient part of Theorem \ref{main_thm}.

\begin{proof}[Proof of Theorem \ref{main_thm}\eqref{main_thm_b}]
From Lemma \ref{lem:l_dom_stoc_e}, we have $F_{l}\leq F_{\gamma}$.
Theorem \ref{thm:schramm_thm} implies that 
\begin{equation*}
l(\T,\rho)=\liminf_{n\to\infty}\frac{E_n(\T,\rho)}{n}\leq \limsup_{n\to\infty}\frac{E_n(\T,\rho)}{n}\leq\gamma(\T),\ \rgw\text{-a.s.}
\end{equation*}
Putting both parts together we get that
\begin{equation*}
F_{\gamma}\leq F_{\overline{l}}\leq F_{l}\leq F_{\gamma}
\end{equation*}
where $F_{\overline l}$ is the c.d.f.~of $\overline{l}(\T,\rho)=\limsup_{n\to\infty}\frac{E_n(\T,\rho)}{n}$. It holds 
$l(\T,\rho)\leq \overline{l}(\T,\rho)$, $\rgw$-a.s.~and $F_{\overline l}=F_{{l}}$, therefore 
$$\E[\overline l(\T,\rho) -l(\T,\rho)]=0.$$ Now, $\overline l(\T,\rho)-l(\T,\rho)$ is $\rgw$-a.s.~positive. The expectation of a nonnegative random variable can be zero only if it is  a.s.~zero, hence $\overline l(\T,\rho)=l(\T,\rho)$  $\rgw$-a.s. 
Therefore,  the limit $L(\T,\rho) = \lim_{n\to\infty} \frac{E_n(\T,\rho)}{n}$ exists $\rgw$-a.s. Since $L(\T,\rho)=\gamma(\T)$ in distribution and $L(\T,\rho)\leq \gamma(\T)$, $\rgw$-a.s.,
we have
\begin{equation*}
 \lim_{n\to\infty} \frac{E_n(\T,\rho)}{n}=\gamma(\T),\ \rgw\text{-almost surely},
\end{equation*}
which proves the transient part.
\end{proof}

\paragraph{Acknowledgements.}
The authors would like to thank Walter Hochfellner for performing simulations of rotor-router walks on
Galton-Watson trees. We also thank the anonymous referees for carefully reading the first version and for suggesting
various improvements.
The research of Wilfried Huss was supported by the Austrian Science Fund (FWF): P24028-N18
and by the Erwin-Schr\"{o}dinger scholarship (FWF): J3628-N26. The research of Ecaterina Sava-Huss was supported by the Erwin-Schr\"{o}dinger scholarship (FWF): J3575-N26.
Moreover, all three authors were supported by the exchange program Amadeus between Graz and Marseille,
project number FR 11/2014.


\bibliographystyle{mypaperhep}
\bibliography{rotor_walks}

\end{document}